\documentclass[11pt,leqno]{article}
\usepackage{amsthm,amsfonts,amssymb,amsmath,oldgerm}
\usepackage{epsfig}
\numberwithin{equation}{section}
\usepackage[thinlines]{easybmat}

\def\eps{\varepsilon }

\def\eps{\varepsilon}

\newcommand\br{\begin{remark}}
\newcommand\er{\end{remark}}
\newcommand\bp{\begin{pmatrix}}
\newcommand\ep{\end{pmatrix}}
\newcommand\be{\begin{equation}}
\newcommand\ee{\end{equation}}
\newcommand\ba{\begin{equation}\begin{aligned}}
\newcommand\ea{\end{aligned}\end{equation}}

\newcommand\nn{\nonumber}

\newcommand{\bap}{\begin{app}}
\newcommand{\eap}{\end{app}}
\newcommand{\begs}{\begin{exams}}
\newcommand{\eegs}{\end{exams}}
\newcommand{\beg}{\begin{example}}
\newcommand{\eeg}{\end{exaplem}}
\newcommand{\bpr}{\begin{proposition}}
\newcommand{\epr}{\end{proposition}}
\newcommand{\bt}{\begin{theorem}}
\newcommand{\et}{\end{theorem}}
\newcommand{\bc}{\begin{corollary}}
\newcommand{\ec}{\end{corollary}}
\newcommand{\bl}{\begin{lemma}}
\newcommand{\el}{\end{lemma}}
\newcommand{\bd}{\begin{definition}}
\newcommand{\ed}{\end{definition}}
\newcommand{\brs}{\begin{remarks}}
\newcommand{\ers}{\end{remarks}}

\newtheorem{theo}{Theorem}[section]

\newtheorem{lem}[theo]{Lemma}

\newtheorem{rem}[theo]{Remark}

\newtheorem{exams}[theo]{Examples}

\numberwithin{equation}{section}

\newcommand{\CalT}{\mathcal{T}}

\newcommand{\MM}{{\mathbb M}}

\newcommand{\sgn}{\text{\rm sgn}}
\newtheorem{theorem}{Theorem}[section]
\newtheorem{proposition}[theorem]{Proposition}
\newtheorem{corollary}[theorem]{Corollary}
\newtheorem{lemma}[theorem]{Lemma}
\newtheorem{definition}[theorem]{Definition}

\newtheorem{example}[theorem]{Example}
\newtheorem{remark}[theorem]{Remark}

\newcommand\cM{{\mathcal M}}

\newtheorem{thm}{Theorem}
\newtheorem{corr}{Corollary}

\newcommand{\RM}{\mathbb{R}}

\newcommand{\CM}{\mathbb{C}}
\newcommand{\WM}{\,\mbox{\bf W}}

\newcommand{\HM}{\,\mbox{\bf H}}

\newcommand{\tr}{\,\mbox{\rm tr}}

\newcommand{\cn}{\operatorname{cn}}
\newcommand{\dn}{\operatorname{dn}}

\title{
Transverse Instability of Periodic Traveling Waves in the Generalized Kadomtsev-Petviashvili Equation
}

\author{\sc \small
Mathew A. Johnson\thanks{Indiana University, Bloomington, IN 47405;
matjohn@indiana.edu: Research of M.J. was partially supported by an NSF Postdoctoral Fellowship under NSF grant DMS-0902192.}
\and
Kevin Zumbrun\thanks{Indiana University, Bloomington, IN 47405;
kzumbrun@indiana.edu:
Research of K.Z. was partially supported
under NSF grants no. DMS-0300487 and DMS-0801745.
 }}
\begin{document}

\maketitle


\begin{abstract}
In this paper, we investigate the spectral instability of
periodic traveling wave solutions of the
generalized Korteweg-de Vries equation to long wavelength transverse perturbations
in the generalized Kadomtsev-Petviashvili equation.  By analyzing high and low
frequency limits of the appropriate periodic Evans function,
we derive an orientation index which yields sufficient
conditions for such an instability to occur.  This index is geometric in nature and applies
to arbitrary periodic traveling waves with minor smoothness and convexity assumptions
on the nonlinearity.  Using the integrable structure of the ordinary differential equation governing the traveling wave
profiles, we are then able to calculate the resulting orientation index for the elliptic function
solutions of the Korteweg-de Vries and modified Korteweg-de Vries equations.
\end{abstract}




\bigbreak
\section{Introduction}
The Korteweg-de Vries equation
\begin{equation}
u_t=u_{xxx}+uu_x\label{eqn:kdv}
\end{equation}
often arises as an model for one-dimensional long wavelength surface waves propagating
in weakly nonlinear dispersive media, as well as the evolution of weakly nonlinear ion acoustic waves in plasmas \cite{TRR}.
When the assumption that the wave is purely one-dimensional is relaxed to allow for
weak dependence in a transverse direction one is led to a variety of multidimensional generalizations
of the KdV.  One of the most well studied weakly two-dimensional variations of the KdV is
the Kadomtsev-Petviashvili (KP) equation [KP] given by
\begin{equation}
\left(u_t-u_{xxx}-uu_x\right)_x+\sigma u_{yy}=0,\label{eqn:KP}
\end{equation}
where the constant $\sigma=\pm 1$ differentiates between equations with positive ($\sigma=+1$)
and negative ($\sigma=-1$) dispersion and the choice of $\sigma$ depends
on the exact physical phenomenon being described\footnote{Notice
that one can always rescale the $y$ variable to force $\sigma$ to be any (non-zero) real number.  However,
for convenience, we will always assume that $\sigma=\pm 1$.}.  For instance, if $\sigma=+1$
\eqref{eqn:KP} is referred to as the KP-I equation, which can be used to model waves in thin
films with high surface tension, while \eqref{eqn:KP} is called the KP-II equation in when $\sigma=-1$,
which can be used to model water waves with small surface tension.  Various other physical applications utilize equations of the
form \eqref{eqn:KP}, such as the modeling of small amplitude internal waves and in the study of unmagnetized dusty
plasmas with variable dust charge \cite{PJ}.

In many applications, appropriate scaling in the physical parameters introduces a parameter
$\alpha>0$ in the nonlinearity yielding a governing equation of the form
\[
u_t=u_{xxx}+\alpha uu_x.
\]
In neighborhoods of parameter space where $\alpha=0$ one is forced to consider higher
order expansions in the nonlinearity, the most natural being of the form
\[
u_t=u_{xxx}+\beta u^2u_x
\]
where $\beta\neq 0$.  The choice of the sign of $\beta$ must be made depending on
the particular physical situation being studied.  In particular, the choice of $\beta$
clearly determines the structure of the stationary homoclinic/heteroclinic solutions
and hence the case of $\beta>0$ and $\beta<0$ define quite distinct dynamics.  In this case
the modified KP equation
\[
\left(u_t-u_{xxx}-\beta u^2u_x\right)_x+\sigma u_{yy}=0
\]
arises naturally as a weakly two-dimensional generalization.  In the context
of dusty plasmas with variable dust charge, the mKP equation can be derived near the ``critical" density case (see \cite{PJ})
for details.

As our theory will not depend on the
explicit form of the nonlinearity, and to encompass as many physical applications as possible,
we will most often work with the generalized KdV (gKdV) equation
\begin{equation}
u_t=u_{xxx}+f(u)_x\label{gkdv}
\end{equation}
and its weakly two-dimensional variation the generalized KP equation
\begin{equation}
\left(u_t-u_{xxx}-f(u)_x\right)_x+\sigma u_{yy}=0\label{gkp}
\end{equation}
where the nonlinearity $f$ is sufficiently smooth and satisfies general convexity assumptions.  For such nonlinearities,
the gKdV equation admits asymptotically constant traveling solutions, known as solitary waves, as well as traveling waves
which are spatially periodic.  It is the latter case we consider here.  As a solution $u(x,t)$ of \eqref{gkdv} is clearly a $y$-independent solution
of \eqref{gkp}, it seems natural to question the stability of such a solution to perturbations which
have a nontrivial dependence on the transverse (y) direction.  Such a transverse instability analysis
is the subject of the current paper: in particular, we study the spectral stability of a stable $y$-independent
spatially periodic traveling wave solution of the gKdV equation to perturbations which are
co-periodic\footnote{Notice by Floquet theory, one-dimensional spectral instability to co-periodic perturbations
implies spectral instability to localized perturbations.} in $x$ with low frequency oscillations in the transverse direction.
To this end, we will loosely follow the general Evans function approach of Gardner \cite{G1} with additional aspects of analysis from the
more recent work of Johnson \cite{J2}.

The transverse instability of solitary waves of the KdV in the KP equation was first conducted by Kadomtsev and Petviashvili \cite{KP}, where
it was found that such solutions are stable to transverse perturbations in the
case of negative dispersion, while they are unstable to long wavelength transverse
perturbations in the case of positive dispersion (even though they are stable in the corresponding one-dimensional problem).  Moreover, in \cite{APS} it was shown
that a short wavelength cutoff for instability exists for the positive dispersion case and the dominate mode of instability was identified.
Other authors have devised various techniques to demonstrate the transverse instability of KdV solitary waves in the KP-I equation: for example,
see the pioneering work of Zakharov \cite{Za}, where the author utilizes the integrability of the KP-I equation via the inverse scattering
transform, and the recent work of Rousset and Tzvetkov \cite{RT3}, where the authors use general PDE techniques\footnote{The methods of Rousset and Tzvetkov provide nonlinear transverse instability of the KdV solitary wave, and has
the advantage over that of Zakharov of generalizing to the full water wave problem in the presence of surface tension.}.
The transverse instability of a one-dimensionally stable gKdV solitary wave in the corresponding gKP equation
has recently been considered in \cite{KTN} using perturbation analysis similar to that of Kadomtsev and Petviashvili.  In particular, 
multiple scale analysis was used to derive an evolution equation for the wave velocity to describe the slow-time response
of the solitary wave in response to the long wavelength transverse perturbations.
The authors conclude that for positive dispersion the solitary waves of the gKdV equation are always unstable to long wavelength
transverse perturbations in the gKP equation.  
Moreover, it was found that for some nonlinearities the solitary waves may in fact be unstable in the case of negative
dispersion.

In the case where the background solution of the gKdV is spatially periodic, there do not seem to be
any results concerning the transverse instability in the gKP equation and, moreover, the stability of such
solutions in general is much less understood.  
This reflects the fact that the spectrum of the corresponding linearized operators
is purely continuous, and hence it seems more difficult for nonlinear periodic waves of the gKdV to be stable than their solitary
wave counterparts.  Moreover, the periodic waves of the gKdV in general have a much more rich structure
than the solitary waves: even in the case of power-law nonlinearities it is not possible to write down a
general elementary representative for all periodic solutions of the gKdV, which stands in contrast to the
solitary wave theory.  Nevertheless, there has been much study recently into the one-dimensional stability of such solutions
under the gKdV flow with results ranging from spectral stability to localized perturbations (see \cite{BD}, \cite{BrJ}, and \cite{HK}) to
nonlinear (orbital) stability to co-periodic perturbations (see \cite{BrJK}, \cite{DK}, and \cite{J1}).
We also want to stress the fact that the periodic solutions are also 
physically relevant as they are used to model nonlinear wave trains: such patterns are prevalent in a variety of applications and
their instabilities have been studied extensively in the literature (see for example the classic works of
Benjamin \cite{Be2}, Benjamin and Feir \cite{BF}, Lighthill \cite{L}, and Whitham\cite{W}.).

The goal of this paper is to perform a transverse instability study
in the case of spatially periodic traveling wave solutions which are stable to perturbations
in the $x$ direction of unidirectional propagation.  To this end, we will develop a somewhat nonstandard orientation index
which detects instability of a $T$-periodic traveling wave of the gKdV to perturbations in the gKP equation
which are $T$-periodic in the $x$-direction with low frequency oscillations in the transverse spatial variable $y$.
Throughout this paper, we will refer to such perturbations as long wavelength transverse perturbations.  This is accomplished
by studying the behavior of the corresponding periodic Evans function in both the high (real spectral) frequency and low (transverse) frequency
limits when the wave number of the transverse perturbation is small but non-zero.  As we will see, the high frequency
analysis is somewhat delicate due to a degeneracy in the KP equation.  In particular,
it is seen that one must take into account not only the higher order effects of the dependence of
the limiting asymptotic ordinary differential equation on the background solution as in \cite{MaZ3} and \cite{PZ},
but also the inherent averaging/cancellation effects due to the periodicity.  Such a result seems new to the literature,
and relies on the use of block-triangularizing transformations as used in \cite{HLZ}.  Moreover, in
the appendix we give a slight simplification of the of the tracking lemmas utilized in \cite{MaZ3}, \cite{PZ},
and \cite{HLZ} by formulating the result in terms of conjugating transformations rather than invariant graphs.
This view point seems more closely related to the style and needs of current research
(see for example the related treatments in \cite{Z6} and \cite{NZ}).  As a result, we will find
that the limiting behavior of the \emph{sign} of the periodic Evans function for large (real) spectral frequency is governed precisely by
the sign of the dispersion parameter $\sigma$.  Thus, the stability of our periodic traveling waves
depends directly on the case of dispersion chosen in the gKP equation, and hence on the exact physical
phenomenon being described.

The low frequency analysis utilizes matrix perturbation theory and the methods of \cite{BrJ} and \cite{J2}.
In particular, we will explicitly construct four linearly independent stationary solutions of the linearization of the
governing PDE: three of these will be given to us by variations in the traveling wave parameters, which we will use
to parameterize the traveling wave solutions of \eqref{gkdv}, while the fourth can be constructed using standard
techniques.  Using variation of parameters then, we can compute the leading order variation of these four
functions in the transverse wave number $k$ for $|k|\ll 1$, and hence can determine the leading order
variation of the periodic Evans function in the transverse wave number at the origin in the spectral plane.
This leading order variation is expressed as a Jacobian determinant relating to the ability to parameterize
near by periodic waves of fixed wavespeeds by the conserved quantities of the gKdV flow.  It follows then that
if this determinant has the opposite sign of that of the high frequency limit (the parameter $\sigma$), we immediately
have a spectral instability to long wavelength transverse perturbations in the gKP equation.  Notice this approach
is somewhat different than that of \cite{J2}, where transverse instability in the generalized Zakharov-Kuznetsov equation
was established by finding sufficient conditions for eigenvalues bifurcating from the origin (in $k$) to enter the unstable
half plane.  In that case, the low frequency analysis involved considering variations of elements of the kernel
of the linearized operator in both the spectral variable and in the transverse wave number.  As a result,
the results in \cite{J2} concern spectral instability to long wavelength perturbations in \emph{both} the
direction of propagation of the background solution and the transverse direction.  The orientation index
derived in our case thus seems to be nonstandard, in the sense that it detects instabilities
to perturbations which have bounded period in the $x$-direction and admit slow modulations in the transverse
direction.

The outline of this paper is as follows.  In section 2, we review the basic properties
of the periodic traveling wave solutions of the gKdV equation \eqref{gkdv}.  In particular, we will discuss
a parametrization of the traveling wave solutions of \eqref{gkdv} which will be useful throughout
our analysis.  In section 3, we conduct our transverse instability analysis
by first conducting a high frequency analysis of the associated periodic Evans function and then
conducting the corresponding low frequency analysis.  As a result, we will have an instability index
which guarantees that a periodic traveling wave solution of the gKdV is spectrally unstable to long wavelength
transverse perturbations in the gKP equation.  This index can be calculated exactly in several model cases and the
relevant results will be given.  In section 4, we end with some closing remarks and in the appendix
we present a proof of the tracking lemma utilized in the high frequency analysis and also
outline the proof of a formula crucial to the low frequency analysis.

\section{Properties of Periodic Traveling GKdV Waves}

In this section, we review some of the basic properties of the periodic traveling wave solutions
of the gKdV equation.  For more details, see \cite{BrJ} or \cite{J1}.  For each $c>0$, a traveling wave with speed $c$ is a solution
of the ordinary differential equation
\begin{equation}
u_{xxx}+f(u)_x-cu_x=0,\label{travelode}
\end{equation}
i.e. they are stationary solutions of \eqref{gkdv} in the moving coordinate frame defined by $x+ct$.
This equation is clearly Hamiltonian and hence we can reduce it to quadrature.  Indeed, by
integrating \eqref{travelode} twice we see that a traveling wave profile of the gKdV must
satisfy the nonlinear oscillator equation
\begin{equation}
\frac{u_x^2}{2}=E+au+\frac{c}{2}u^2-F(u),\label{quad1}
\end{equation}
where $F$ is an antiderivative of the nonlinearity $f$ satisfying $F(0)=0$ and
$a$ and $E$ are constants of integration.  Thus, the traveling waves form a four
parameter family of solutions of \eqref{gkdv} described by the constants $a$, $E$, and $c$
together with a fourth constant of integration corresponding to a translation mode: this
translation direction is simply inherited from the translation invariance of \eqref{gkdv}
and hence can be modded out.  It follows that on open subsets of $\RM^3=(a,E,c)$ equation \eqref{quad1}
admits a periodic orbit.  Moreover, the boundary of these open subsets correspond to
solutions which decay asymptotically at infinity and hence can be identified with
the solitary wave solutions.  In particular, notice that in order for \eqref{quad1} to
admit a solitary wave solution, the constant $E$ must be fixed by the prescribed boundary conditions
and we must have $a=0$.  It follows the solitary waves form a codimension two
subset of the family of traveling waves.

In general, equation \eqref{gkdv} admits three conserved quantities.  In order to define these, let
\[
V(u;a,c)=F(u)-au-\frac{c}{2}u^2
\]
be the effective potential arising in the nonlinear oscillator equation \eqref{quad1}.  Throughout
this paper, we will assume the roots $u_{\pm}$ of the equation $E=V(u;a,c)$ are simple, satisfy
$u_-<u_+$, and that $V(u;a,c)<E$ for $u\in(u_-,u_+)$.  As a consequence, $u_{\pm}$ are $C^1$
functions of the traveling wave parameters $a$, $E$, and $c$ and, without loss of generality,
we can set $u(0)=u_{-}$.  It follows that we can express the period of a periodic solution
of \eqref{travelode} via the formula
\[
T=T(a,E,c)=\sqrt{2}\int_{u_{-}}^{u_{+}}\frac{du}{\sqrt{E-V(u;a,c)}}.
\]
By a standard procedure, the above integral can be regularized at the square root
branch points and hence represents a $C^1$ function of $a$, $E$, and $c$.  Similarly,
the conserved quantities of the gKdV flow can be represented as
\begin{align*}
M(a,E,c) &=  \int_0^T u(x)\; dx = 2\int_{u_-}^{u_+} \frac{u\; du}{\sqrt{2\left(E-V(u;a,c)\right)}}\\
P(a,E,c) &= \int_0^T u^2(x)\; dx = 2\int_{u_-}^{u_+} \frac{u^2\; du}{\sqrt{2\left(E-V(u;a,c)\right)}}\\
H(a,E,c) &=  \int_0^T\left( \frac{u_x^2}{2} - F(u)\right)dx =  2\int_{u_-}^{u_+} \frac{E-V(u;a,c) - F(u)}{\sqrt{2\left(E-V(u;a,c)\right)}}\;du.
\end{align*}
representing the mass, momentum, and Hamiltonian, respectively.  As above, these integrals can be regularized at the branch points
and hence represent $C^1$ functions of the traveling wave parameters.  As we will see, the gradients of the period and mass
of the solution $u$  will play a very large role in this paper.  However, as pointed out in \cite{BrJ}, when $E\neq 0$ gradients in the period
can be interchanged for gradients of the conserved quantities via the relation
\[
E \nabla_{a,E,c} T + a \nabla_{a,E,c} M + \frac{c}{2} \nabla_{a,E,c} P + \nabla_{a,E,c} H = 0.
\]
where $\nabla_{a,E,c}=\left<\partial_a,\partial_E,\partial_c\right>$.  Thus, all gradients involved in the
results of this paper can be expressed \textit{completely} in terms of the gradients of the conserved
quantities of the gKdV flow, which seems to be desired from a physical point of view.

We now discuss our parametrization of the family of periodic traveling wave solutions of \eqref{gkdv} more carefully.
A major technical assumption throughout this paper is that the period
and mass provide good local coordinates for the periodic traveling waves of fixed wave speed $c>0$.
More precisely, given a periodic traveling wave $u(\cdot;a_0,E_0,c_0)$ of \eqref{gkdv} with $c_0>0$ we assume the map
\[
(a,E)\mapsto\left(T(a,E,c_0),M(a,E,c_0)\right)
\]
have a unique $C^1$ inverse in a neighborhood of $(a_0,E_0)\in\RM^2$, which is clearly
equivalent with the non-vanishing of the Jacobian determinant
\[
\{T,M\}_{a,E}:=\det\left(\frac{\partial(T,M)}{\partial(a,E)}\right)
\]
at the point $(a_0,E_0,c_0)$.
As we will see, the sign of this Jacobian controls the low frequency analysis presented in this paper.
It is worth mentioning that while this may seem like a rather obscure requirement, this Jacobian
has already shown to be important in the stability theory of periodic traveling waves of the gKdV: see
for example \cite{BrJ}, \cite{BrJK}, \cite{J1}, and \cite{J2}.  In particular, this Jacobian
been computed in \cite{BrJK} for several power-law nonlinearities and, in these cases, has been
shown to be generically non-zero.  Moreover, such a non-degeneracy condition should not be
surprising: a similar non-degeneracy condition must often be enforced in the stability
theory for solitary waves (see \cite{Bo}, \cite{Be2}, and \cite{PW}).

\section{Transverse Instability Analysis}

We now begin our stability analysis.  Let $u=u(\cdot;a,E,c)$ be a $T=T(a,E,c)$-periodic
traveling wave solution of \eqref{gkdv}.  Moreover, we assume that $u$ is a stable solution
of the one-dimensional gKdV equation\footnote{Else the issue of transverse instability is of no interest, as instabilities
to unidirectional perturbations will prevent stability to higher dimensional perturbations.}.  As noted in the introduction, it is clear then that $u$ is a $y$-independent
solution of the generalized KP equation \eqref{gkp} for either $\sigma=\pm 1$.  We are interested in the spectral
stability of $u$ as a solution of \eqref{gkp} to small perturbations.  To this end, consider a small perturbation
of $u$ of the form
\[
\psi(x,y,t)=u(x)+\eps v(x,y,t)+\mathcal{O}(\varepsilon^2),~~|\eps|\ll 1
\]
where $v(\cdot,y,t)\in L^2(\RM)$ for each $(y,t)\in\RM^2$ and $v(x,\cdot,t)\in L^\infty(\RM)$ for each $(x,t)\in\RM^2$.
Forcing $\psi$ to solve the traveling gKP equation
\begin{equation}\label{travelgkp}
\left(u_t-u_{xxx}-f(u)_x+cu_{x}\right)_x+\sigma u_{yy}=0,
\end{equation}
yields a hierarchy of consistency conditions.  The $\mathcal{O}(\eps^0)$ equation
clearly holds since $u$ solves \eqref{travelgkp}, and the $\mathcal{O}(\eps^1)$ equation reads as
\[
\partial_x\left(\partial_t+\partial_x\mathcal{L}[u]\right)v+\sigma v_{yy}=0
\]
where $\mathcal{L}[u]=-\partial_x^2-f'(u)+c$ is a periodic Hill operator.
As this linearized equation is autonomous in both time and the spatial variable $y$, we may seek separated solutions
of the form
\[
v(x,y,t)=e^{-\mu t+iky}v(x)
\]
where $\mu\in\CM$, $k\in\RM$, and $v\in L^2(\RM)$.  This leads one to the (generalized) spectral problem
\begin{equation}
\left(\partial_x^2\mathcal{L}[u]-\sigma k^2\right)v=\mu \partial_x v\label{spec}
\end{equation}
considered on the real Hilbert space $L^2(\RM)$.  We refer to the background solution $u$ as being spectrally stable
in $L^2(\RM)$ if \eqref{spec} has no $L^2(\RM)$ spectrum with\footnote{Usually, one defines spectral stability as the absence
of spectrum with positive real part.  However, in our case the spectrum is symmetric about the imaginary axis and hence
spectral stability is equivalent with the spectrum being confined to the imaginary axis.} $\Re(\mu)\neq 0$ for any $k\in\RM$.

Since the coefficients of the differential operator $\mathcal{L}[u]$ are $T$-periodic, as they depend on the background solution $u$,
standard results in Floquet theory implies the $L^2$ spectrum of \eqref{spec}
is purely continuous, and consists entirely of $L^\infty(\RM)$ eigenvalues.  Indeed, the fact that \eqref{spec}
can have no $L^2(\RM)$ eigenvalues is clear: writing \eqref{spec} a first order system of the form
\[
{\bf Y}_x=\HM(x;\mu,k){\bf Y}
\]
and letting $\Phi(x;\mu,k)$ be a matrix solution satisfying the initial condition $\Phi(0;\mu,k)={\bf I}$ for all
$(\mu,k)\in\CM\times\RM$, we define the monodromy operator, or the period map, to be
\[
\MM(\mu):=\Phi(T;\mu,k).
\]
Notice that given any vector solution ${\bf Y}$ of \eqref{spec}, the monodromy operator is a matrix such that
\[
\MM(\mu){\bf Y}(x;\mu,k)={\bf Y}(x+T;\mu,k)
\]
for all $(x,\mu,k)\in\RM\times\CM\times\RM$.  Assuming now for simplicity that
${\bf Y}$ is an eigenvector of $\MM(\mu)$ with eigenvalue $\lambda$,
we clearly have that
\[
{\bf Y}(NT;\mu,k)=\MM(\mu)^N{\bf Y}(0;\mu,k)=\lambda^N{\bf Y}(0,\mu,k).
\]
Thus, if ${\bf Y}(x;\mu,k)$ decays as $x\to\infty$ it must become unbounded as $x\to-\infty$.  Thus the best we can hope for
is for ${\bf Y}(x;\mu,k)$ to remain bounded on $\RM$, which corresponds in this example to $\lambda\in S^1$, i.e. $|\lambda|=1$.
For more details, see \cite{H} for example.

Following Gardner (see \cite{G1} and \cite{G2}), we define the periodic Evans function for our problem to be
\[
D\left(\mu,k,\lambda\right)=\det\left(\MM(\mu,k)-\lambda{\bf I}\right),~~(\mu,k,\lambda)\in\CM\times\RM\times\CM.
\]
The complex constant $\lambda$ is called the Floquet multiplier and is related to the class of
admissible perturbations in \eqref{spec}.  In particular, notice that $\lambda=1$ corresponds
to $T$-periodic perturbations of the background solution $u$.
Clearly, $D(\mu,k,\lambda)$ is an entire function of $\mu$ and $k$ for each fixed $\lambda\in\CM$
since the coefficient matrix $\HM(x,\mu,k)$ depends as such on $\mu$ and $k$.
This allows an analytical characterization of the $L^2(\RM)$ spectrum of \eqref{spec}:
the generalized spectral problem \eqref{spec} has a non-trivial bounded solution for a given $k\in\RM$
if and only if there exists a $\kappa\in\RM$ such that
\[
D(\mu,k,e^{i\kappa})=0.
\]
In particular, $D(\mu,k,1)=0$ if and only if \eqref{spec} has a non-trivial bounded solution in $L^2_{\rm per}([0,T])$
for a given $k\in\RM$.  Moreover, the following property will be useful in the low frequency analysis conducted later in the
paper.

\begin{lem}\label{gKP:EvansLem}
The function $D(\mu,k,\lambda)$ is an even function of both $\mu$ and $k$.
\end{lem}

\begin{proof}
Since the spectral problem $\left(\partial_x^2\mathcal{L}[u]+\sigma k^2\right)v=\mu\partial_x v$
is invariant under the transformation $k\mapsto -k$, it follows that $D(\mu,k,\lambda)$ is an even function
of $k$.  To analyze the parity in $\mu$, we write
\begin{align*}
D(\mu,k,\lambda)&=\det\left(\MM(\mu,k)-\lambda{\bf I}\right)\\
&=\lambda^4+a(\mu,k)\lambda^3+b(\mu,k)\lambda^2+c(\mu,k)\lambda+1
\end{align*}
where $a(\mu,k)=-\tr\left(\MM(\mu,k)\right)$ and $b(\mu,k)=\frac{1}{2}\left(\tr(\MM(\mu,k)^2)-\tr(\MM(\mu,k))^2\right)$.
Since the spectral problem is invariant under the transformation $x\to -x$ and $\mu\to -\mu$, it follows that the matricies
$\MM(\mu,k)$ and $\MM(-\mu,k)^{-1}$ are similar and hence a direct calculation yields
\begin{align*}
D(\mu,k,\lambda)&=\lambda^4\det\left(\MM(-\mu,k)-\frac{1}{\lambda}{\bf I}\right)\\
&=\lambda^4+c(-\mu,k)\lambda^3+b(-\mu,k)\lambda^2+a(-\mu,k)\lambda+1.
\end{align*}
It follows that $c(\mu,k)=a(-\mu,k)$ and $b(\mu,k)=b(-\mu,k)$, and hence
\[
D(\mu,k,1)=2+a(\mu,k)+a(-\mu,k)+\frac{b(\mu,k)+b(-\mu,k)}{2}.
\]
Thus, $D(\mu,k,1)$ is an even function of $\mu$.
\end{proof}

The goal of our analysis is to provide sufficient conditions to ensure that when $0<|k|\ll 1$ the function $\mu\mapsto D(\mu,k,1)$ has a
non-zero real root, corresponding to an exponential instability of the underlying wave.  To this end, we derive
an orientation index by comparing the high frequency and low frequency (in $\mu$) asymptotics of the function $D(\mu,k,1)$.
In particular, we will see that the sign of $D(\mu,k,1)$ for large real $\mu$ equals the sign of $\sigma$.  Thus,
if the quantity $D(0,k,1)$ has the opposite sign of $\sigma$, we can infer the existence of a $\mu^*\in\RM^+$ such that
$D(\mu^*,k,1)=0$ which implies exponential instability of the background solution.  We begin by analyzing the high frequency
behavior of the periodic Evans function.

%

\subsection{High Frequency Limit}

In this section, we study the large real $\mu$ behavior of the periodic Evans function $D(\mu,k,1)$
when $k\neq 0$.  To begin, rescale \eqref{spec} with the change of variables $\tilde{x}=|\mu|^{1/3}x$
to obtain (after dropping the tilde's) the spectral problem
\begin{equation}
\left(-\partial_x^4-|\mu|^{-2/3}\partial_x^2\left(f'(u)+c\right)-\sigma k^2|\mu|^{-4/3}\right)v=\partial_x v.\label{rescale}
\end{equation}
This can be rewritten as a first order system of the form
\begin{equation}
{\bf W}'=\underbrace{\left(
                   \begin{array}{cccc}
                     0 & 1 & 0 & 0 \\
                     0 & 0 & 1 & 0 \\
                     0 & 0 & 0 & 1 \\
                     0 & -1 & 0 & 0 \\
                   \end{array}
                 \right)}_{{\bf H}_0(\mu)}{\bf W}
+\underbrace{\left(
                   \begin{array}{cccc}
                     0 & 0 & 0 & 0 \\
                     0 & 0 & 0 & 0 \\
                     0 & 0 & 0 & 0 \\
                     \chi & A_1|\mu|^{-2/3} & A_2|\mu|^{-2/3} & 0 \\
                   \end{array}
                 \right)}_{{\bf B}(\mu)}{\bf W},\label{rescalesystem}
\end{equation}
where $A_1=-2f''(u)u_x$, $A_2=-f'(u)+c$ and
\be\nn
\chi= \frac{1}{2}A_{1,x}|\mu|^{-2/3}-\sigma k^2 |\mu|^{-4/3} .
\ee
On a heuristic level then, we expect that the monodromy operator for $\mu\gg 1$ will behave like
\[
\MM(\mu)\approx e^{{\bf H}_0|\mu|^{1/3}T}
\]
and hence
\[
D(\mu,k,1)\approx\det\left(e^{{\bf H}_0|\mu|^{1/3}T}-{\bf I}\right)
\]
However, the matrix ${\bf H}_0$ clearly has an eigenvalue of $0$ and hence this heuristic argument
leads us to expect that $D(\mu,k,1)\to 0$ as $\mu\to+\infty$.  From the point of view of an orientation index,
this does not provide us with sufficient information: we must know what the limiting \textit{sign} of the
Evans function is.  Thus, although we expect $D(\mu,k,1)$ vanishes in the limit as $\mu\to\infty$, we must
analyze the situation more closely to determine if the \textit{sign} of $D(\mu,k,1)$ has a limiting value.
This is the content of the following lemma.

\begin{lem}\label{hfLem}
For $k\neq 0$, we have the high frequency limit
\[
\lim_{\mu\to\pm\infty}\sgn\left( D(\mu,k,1)\right)=\sgn(\sigma).
\]
\end{lem}

\begin{rem}
Notice this result is somewhat unexpected due to the form of the rescaled equation \eqref{rescale} since
the term $\sigma k^2$ enters at only {\em second} order in $|\mu|^{-2/3}$.  Thus, this is a {\em very}
small term compared with the $\mathcal{O}(|\mu|^{-2/3})$ terms involved.  However, the following proof will
show that, upon averaging, the periodicity of the underlying solution implies cancelation of the lower order effects
and hence the asymptotic behavior for large $\mu$ must be determined by the higher order terms.
\end{rem}

\begin{proof}
First, notice that by Lemma \ref{gKP:EvansLem} it is enough to consider the limit as $\mu\to+\infty$ only.
Consider the rescaled first order system \eqref{rescalesystem} and define $\eps=|\mu|^{-2/3}\ll 1$.
We refer to the constant matrix ${\bf H}_0$
as the principle part and regard the matrix ${\bf B}$ as a matrix of error terms.  We begin by diagonalizing the
principle part.  To this end, let $\lambda=\frac{1}{2}\left(1+i\sqrt{3}\right)$ and notice if we define the matrix
\[
{\bf Q}:=\bp
         -1 & -1 & -1 & 1\\
         1 & -\lambda & -\lambda^* & 0\\
         -1 & \lambda^* & \lambda & 0\\
         1 & 1 & 1 & 0\\
         \ep,
\]
then a straight forward calculation yields
\[
{\bf Q}^{-1}{\bf H_0}{\bf Q}=\bp
                             -1 & 0 & 0 & 0\\
                             0 & \lambda & 0 & 0\\
                             0 & 0 & \lambda^* & 0\\
                             0 & 0 & 0 & 0\\
                             \ep.
\]
Note that $\Re \lambda=\Re \lambda^*$ is positive,
hence the real parts of the diagonal entries of the stable, neutral,
and unstable diagonal blocks of ${\bf Q}^{-1}{\bf H_0}{\bf Q}$ each have a spectral
gap, one from the other.  Using the change coordinates ${\bf W}={\bf QY}$ in \eqref{rescalesystem}, we can consider the
first order system
\[
{\bf Y}'=\left[
   \bp
   -1 & 0 & 0 & 0\\
   0 & \lambda & 0 & 0\\
   0 & 0 & \lambda^* & 0\\
   0 & 0 & 0 & 0\\
   \ep
+
   \underbrace{{\bf Q}^{-1}{\bf BQ}}_{\bf \tilde{B}}
   \right]{\bf Y}
\]
By a brief but tedious calculation, we see that the matrix ${\bf \tilde{B}}$ take the block form
\[
{\bf \tilde{B}}=\left(
\begin{BMAT}{c.c}{c.c}
\mathcal{O}(\eps)  &  \begin{BMAT}{c}{c.c.c}
                                \frac{1}{3}\chi\\
                                \frac{1}{3}\chi\\
                                \frac{1}{3}\chi
                             \end{BMAT}
  \\
\begin{BMAT}{c.c.c}{c}
(A_1-A_2)\eps-\chi & (-\lambda A_1+\lambda^*A_2)\eps-\chi & (-\lambda^*A_1+\lambda A_2)\eps-\chi
\end{BMAT}
& \chi
\end{BMAT}
\right)
\]
where the upper left hand block is $3\times 3$ and the lower right hand block is $1\times 1$.
Now, define the $T$-periodic matrix valued function
\[
{\bf S}={\bf I}_4+\eps\left(
\begin{BMAT}{c.c}{c.c}
0& ~0~\\
  \begin{BMAT}{c.c.c}{c}
  -A_1+A_2+\frac{1}{2}A_{1,x} & -A_1+\frac{\lambda^*}{\lambda}A_2+\frac{1}{2\lambda}A_{1,x} & -A_1+\frac{\lambda}{\lambda*}A_2+\frac{1}{2\lambda*}A_{1,x}
  \end{BMAT}
& ~0~
\end{BMAT}
\right)
\]
where ${\bf I}_4$ is the standard $4\times 4$ identity matrix and again the upper left hand block is $3\times 3$.
Then another straightforward computation implies that
\[
{\bf S}^{-1}{\bf\tilde{B}}{\bf S}=
\left(
  \begin{BMAT}{c.c}{c.c}
  \mathcal{O}(\eps) & \mathcal{O}(\eps)\\
  \mathcal{O}(\eps^3) & \frac{1}{2}A_{1,x}\eps+\eps^2\left(\frac{1}{2}A_1 A_{1,x}-\sigma k^2\right)
  \end{BMAT}
\right)
\]
where again the upper left hand block is $3\times 3$.

Now, noticing that ${\bf S}'=\mathcal{O}(\eps^{3/2})$ and that ${\bf S}$ is a $\mathcal{O}(\eps)$ perturbation of the identity,
it follows that making the variable coefficient change
of variables ${\bf U}={\bf SY}$ yields a first order system of the form
\[
{\bf U}_x=\left[{\bf Q}^{-1}{\bf H}_0{\bf Q}+\left(
  \begin{BMAT}{c.c}{c.c}
  \mathcal{O}(\eps) & \mathcal{O}(\eps)\\
  \mathcal{O}(\eps^{3/2}) & \frac{1}{2}A_{1,x}\eps+\eps^2\left(\frac{1}{2}A_1 A_{1,x}-\sigma k^2\right)
  \end{BMAT}
\right)\right]{\bf U}.
\]
In particular, we see that the resulting coefficient
matrix is approximately block upper triangular with error of order $\mathcal{O}(\eps^{3/2})=\mathcal{O}(|\mu|^{-1})$.
Applying Lemma \ref{reduction} from Appendix A (notice that here $\eta=1$, $N=\mathcal{O}(\eps)$, and $\delta=\mathcal{O}(\eps^{3/2})$) and Remark \ref{periodic} we find that
there is a $T$-periodic 
change of coordinates ${\bf X}={\bf ZU}$ of the form
\[
{\bf Z}=\begin{pmatrix}
I_3 & 0 \\
\Phi & 1 \\
\end{pmatrix},
\]
where $\Phi=O(\eps^{3/2})$ is of dimension $1\times 3$, taking the system to an exact upper block triangular form with
diagonal blocks
\[
-1+\mathcal{O}(\eps),\quad \begin{pmatrix} \lambda & 0\\ 0 & \lambda^* \\ \end{pmatrix} +O(\eps),
\quad \hbox{\rm and }\; \frac{1}{2}A_{1,x}\eps+\eps^2\left(\frac{1}{2}A_1 A_{1,x}-\sigma k^2\right)+\mathcal{O}(\eps^{5/2}).
\]

Finally, by the block-triangular form {\it plus periodicity of the coordinate
changes} we may compute the periodic
Evans function as the product of the periodic Evans function
of the diagonal blocks of this transformed system, integrated
over a period $\tilde T:=T|\mu|^{1/3}$ going to infinity.
The contribution from the stable block is\footnote{ Notice each of the three subsystems gives real value,
since they clearly have real coefficients.} approximately $e^{-|\mu|^{1/3}T}-1$,
so has sign $-1$. Similarly, the unstable block gives a positive sign.
The third block, corresponding to the neutral block, gives approximately
\begin{align*}
\exp\left(\int_0^{|\mu|^{1/3}T}\left(\frac{1}{2}A_{1,x}\eps+ \eps^2\left(\frac{1}{2}A_1 A_{1,x}-\sigma k^2\right)\right)(s)ds\right)&-1\\
&=\exp\left(-\sigma k^2|\mu|^{-1}T\right)-1\\
&\sim -\sigma k^2|\mu|^{-1}
\end{align*}
and hence has the opposite sign as the dispersion parameter $\sigma$.  Notice, here we have used heavily the fact
that the periodicity of the background solution implies
\[
\int_0^{|\mu|^{1/3}T}A_{1,x}(x)dx=0~~\textrm{and}~~\int_0^{|\mu|^{1/3}T}A_1(x)A_{1,x}(x)dx=0
\]
and hence yields cancelation in averaging of the lower order terms.   Combining these results, we find that the periodic Evans function
has sign $\sgn(\sigma)$ for $k\neq 0$ and $\mu\gg 1$ as claimed.
\end{proof}

Notice that one could rework the above proof in the case where the background solution corresponds to a homoclinic
orbit of the traveling wave ODE \eqref{travelode}.  As a result, Lemma \ref{hfLem} still holds in the solitary wave
setting.  This result seems to be new to the literature, and may give valuable insight to the transverse instability
analysis in the solitary wave case.

\subsection{Low Frequency Analysis and Instability}

Now that we have a handle on the limiting sign of $D(\mu,k,1)$ as $\mu\to\pm\infty$, we turn our attention to determining
the sign of the quantity $D(0,k,1)$ for $|k|\ll 1$.  For, once we have this information we see that the negativity of
the orientation index
\[
\sigma\cdot\sgn\left(D(0,k,1)\right)
\]
provides a sufficient condition for transverse instability of the underlying periodic wave $u$.
To this end, we utilize the methods of \cite{J2}, which builds off the methods of \cite{BrJ}, to derive
an asymptotic expansion of the function $D(0,k,1)$ for $|k|\ll 1$.  However, it should be
pointed out that the analysis in this case is seemingly more delicate than that
considered by \cite{J2} due to the fact that the
unperturbed spectral problem at $\mu=0$, i.e. 
\begin{equation}
\partial_x^2\mathcal{L}[u]v=0,\label{nullnp}
\end{equation}
which does not define a Hamiltonian equation.  In particular, this equation is not reducible to quadrature
and hence
one can not use integrability of the corresponding traveling wave ODE
\[
\left(-u_{xxx}-f(u)_x+cu_x\right)_x=0
\]
to construct a basis for the monodromy operator at $\mu=0$, $k=0$.  However, one can use the integrability of
the traveling wave ODE \eqref{travelode} for the gKdV equation to construct\footnote{
Notice that here we are considering the traveling wave ODE
as a formal differential equation with out any reference to boundary conditions.}
three linearly independent solutions of equation \eqref{nullnp}: namely, the functions $u_x$, $u_a$, and $u_E$
are easily seen to satisfy
\[
\mathcal{L}[u]u_x=0,~~\mathcal{L}[u]u_E=0,~~\mathcal{L}[u]u_a=-1
\]
(see \cite{BrJ} for details) and hence provide a basis of solutions for the differential equation
\[
\partial_x\mathcal{L}[u]v=0.
  \]
Thus, we are missing one null direction of the (formal) operator $\partial_x^2\mathcal{L}[u]$:
we need a solution to the equation $\mathcal{L}[u]v=x$.  However,
notice that the functions $u_x$ and $u_E$ provide two linearly independent solutions of
the differential equation\footnote{Here, again, we consider the formal operator $\mathcal{L}[u]$ without any reference to boundary conditions}
$\mathcal{L}[u]v=0$, and hence
one can use variation of parameters 
to solve the nonhomogeneous equation.  After a straightforward calculation
it is seen that
\[
\phi(x):=\left(\int_0^x su_E(s)ds\right)u_x(x)-\left(\int_0^x su_s(s)ds\right)u_E(x)
\]
is linearly independent from $u_x$, $u_a$, and $u_E$ and satisfies $\mathcal{L}[u]\phi=x$, 
and hence we can use these four functions to construct the corresponding monodromy matrix at the origin.
Using perturbation theory then, we should be able to determine how these functions bifurcate as $k$ varies but remains very small and hence
be able to determine a leading order expansion of the periodic Evans function near $k=0$.
This is the content of the following lemma.

\begin{lem}\label{lfLem}
The following asymptotic relation holds in a neighborhood of $k=0$:
\[
D(0,k,1)=-\left(PT-M^2\right)\{T,M\}_{a,E}\left(\sigma k^2\right)^2+\mathcal{O}(k^6).
\]
\end{lem}

\begin{proof}
We begin by writing the linearized equation \eqref{spec}
as a first order system with coefficient matrix
\[
\HM(x,\mu,k)=\left(
               \begin{array}{cccc}
                 0 & 1 & 0 & 0 \\
                 0 & 0 & 1 & 0 \\
                 0 & 0 & 0 & 1 \\
                -\sigma k^2-f'''(u)u_x^2-f''(u)u_{xx} & -2f''(u)u_x-\mu & -f'(u)+c & 0 \\
               \end{array}
             \right).
\]
We now define the matrix $\WM(x,\mu,k)$ as the matrix solution of the first order system $Y'=\HM(x,\mu,k)Y$
such that
\begin{equation}
\WM(x,0,0)=\left(
             \begin{array}{cccc}
               u_x & u_a & u_E & \phi \\
               u_{xx} & u_{ax} & u_{Ex} & \phi_x \\
               u_{xxx} & u_{axx} & u_{Exx} & \phi_{xx} \\
               u_{xxxx} & u_{axxx} & u_{Exxx} & \phi_{xxx} \\
             \end{array}
           \right),\label{eqn:solnmatrix}
\end{equation}
and we fix the initial condition $\WM(0,\mu,k)=\WM(0,0,0)$ for all $(\mu,k)\in\CM\times\RM$.  Defining
$\delta\WM(\mu,k)=\WM(x,\mu,k)\big{|}_{x=0}^T$, a straightforward calculation gives
\[
\delta\WM(0,0)=\left(
                 \begin{array}{cccc}
                   0 & 0 & 0 & -\frac{\partial u_{-}}{\partial E}\int_0^Txu_x(x)dx \\
                   0 & V'(u_{-})T_a & V'(u_{-})T_E & -V'(u_{-})\int_0^Txu_E(x)dx \\
                   0 & 0 & 0 & -T+V''(u_{-})\frac{\partial u_{-}}{\partial E}\int_0^Txu_x(x)dx \\
                   0 & -V'(u_{-})V''(u_{-})T_a & -V'(u_{-})V''(u_{-})T_E & V''(u_{-})V'(u_{-})\int_0^Txu_E(x)dx \\
                 \end{array}
               \right).
\]
Since $\delta\WM(0,k)$ is analytic in $k^2$ by Lemma \ref{gKP:EvansLem}, 
it follows that
%
\[
D(0,k,1)=\frac{\det(\delta\WM(0,k))}{\det(\WM(0,0,1))}=\mathcal{O}(k^4)
\]
for $|k|\ll 1$. 
Moreover, since the first column of the matrix $\delta\WM(0,k)$ 
is $\mathcal{O}(k^2)$, one can also easily see that
the $\mathcal{O}(k^2)$ variation in the $\phi$ direction contributes to terms in $D(0,k,1)$ of order $\mathcal{O}(k^6)$ near $k=0$.  Thus,
in order to compute the $\mathcal{O}(k^4)$ variation in the Evans function, we need only compute
the $\mathcal{O}(k^2)$ variation in the $u_x$, $u_a$, and $u_E$ directions\footnote{Note that
finding a useful expression for the $\mathcal{O}(k^2)$ variation in the $\phi$ direction
would  be a very daunting task, since one would have to apply variation of parameters to solve a problem
with non-homogeneity which was constructed using variation of parameters.}.

Computing the necessary variations 
can be done using the variation of parameters formula.
To this end, we define the vector solutions corresponding to $u_x$, $u_a$, $u_E$ and $\phi$ be given by $Y_1$, $Y_2$, $Y_3$ and $Y_4$,
respectively and define
\[
\frac{\partial}{\partial k^2}Y_j(T,0,k)\big{|}_{k=0}=\WM(T,0,0)\int_0^T\WM(x,0,0)^{-1}\left(Y_j(x)\cdot e_j\right)e_4
                                                                            dx,
\]
where $\cdot$ represents the standard inner product on $\RM^4$,
$e_j$ is the $j^{th}$ column of the identity matrix on $\RM^4$,
and $\WM(T,0,0)=\WM(0,0,0)+\delta\WM(0,0)$ where
\[
\WM(0,0,0)=\left(
             \begin{array}{cccc}
               0 & \frac{\partial u_{-}}{\partial a} & \frac{\partial u_{-}}{\partial E}  & 0 \\
               -V'(u_{-}) & 0 & 0 & 0 \\
               0 & 1-V''(u_{-})\frac{\partial u_{-}}{\partial a}  & -V''(u_{-})\frac{\partial u_{-}}{\partial E}  & 0 \\
               V''(u_{-})V'(u_{-}) & 0 & 0 & -1 \\
             \end{array}
           \right).
\]
Moreover, in Appendix B we will show that
\begin{equation}
\WM(x,0,0)^{-1}e_4
=\left(
   \begin{array}{c}
     -\int_0^x\int_0^su_E(z)dzds \\
     -x \\
     \int_0^xu(s)ds \\
     -1 \\
   \end{array}
 \right),\label{eqn:vinv}
\end{equation}
which allows for straight forward computation of the necessary variations.
Thus, we have
\[
\delta\WM(0,k)=\delta\WM(0,0)-\left(\sum_{j=1}^4\frac{\partial}{\partial k^2}Y_j(T,0,k)\big{|}_{k=0}\otimes e_j\right)\sigma k^2+\mathcal{O}(k^4)
\]
and since $\det(\WM(0,0,0))=1$ it follows that
\[
D(0,k,1)=\det\left(\delta\WM(0,0)-\left(\sum_{j=1}^3\frac{\partial}{\partial k^2}Y_j(T,0,k)\big{|}_{k=0}\otimes e_j\right)
                    \sigma k^2\right)+\mathcal{O}(k^6).
\]

Elementary row operations allow one to simplify the matrix involved in the above determinant.  Indeed, a large simplification
is possible if replaces the third row with the result from multiplying the first row by $V''(u_{-})$ and adding it to the third row,
and similarly replace the fourth row with $V''(u_{-})$ times the second row plus the fourth row.  In particular, these operations
imply the fourth entry in the first column is of order $\mathcal{O}(k^4)$ and hence does not enter into the calculation.
The resulting expressions are still too long to explicitly write out here, but can be easily handled by using a
computer algebra system.  Indeed, a direct calculation yields
\begin{align*}
D(0,k,1)&=\left(Q_1\frac{\partial u_{-}}{\partial E}+Q_2\frac{\partial u_{-}}{\partial a}\right)V'(u_-)\{T,M\}_{a,E}(\sigma k^2)^2
+\mathcal{O}(k^6)
\end{align*}
where $Q_1=(M^2-PT-MTu_-+T^2u_-^2)$ and $Q_2=(MT-T^2u_-)$.  Finally, by noticing from \eqref{quad1} that
\[
V'(u_-)\frac{\partial u_{-}}{\partial E}=1~~\textrm{and}~~V'(u_-)\frac{\partial u_{-}}{\partial a}=u_-
\]
we have
\[
\left(Q_1\frac{\partial u_{-}}{\partial E}+Q_2\frac{\partial u_{-}}{\partial a}\right)V'(u_-)=M^2-PT,
\]
which completes the proof.
\end{proof}

\begin{rem}
In \cite{BrJ}, Bronski and Johnson derive the low-frequency expansion for the Evans function $D_{gKdV}(\mu,\kappa)$ for the
gKdV equation \eqref{gkdv} in the spectral parameter $\mu$.  In particular, it was shown that
\[
D_{gKdV}(\mu,1)=-\frac{1}{2}\det\left(\frac{\partial(T,M,P)}{\partial(a,E,c)}\right)\mu^3+\mathcal{O}(|\mu|^4)
\]
and hence an index which detects exponential instabilities to co-periodic perturbations was derived by comparing the
sign of the above Jacobian determinant to the limiting behavior for $\mu\to+\infty$.  We believe that it would
be interesting and beneficial to derive the corresponding low-frequency expansion for the Evans function $D(\mu,k,\kappa)$
considered in this paper, as it may give a better understanding of the connection (if there is any) between the generation
of transverse instabilities and uni-directional instabilities.  Indeed, such analysis could serve as a starting point
for a corresponding transverse \emph{stability} analysis of periodic waves.
While it seems natural that the above three-by-three Jacobian determinant should control the (spectral) low-frequency
limit, it is not clear how the degeneracy of the gKP equation (corresponding to the extra x-spatial derivative)
affects the situation.  In particular, the required calculations seem to be considerably more complicated than those considered in Lemma \ref{lfLem}
and a useful identification of the leading order behavior for $|\mu|\ll 1$ is yet to be obtained.
\end{rem}

Combining Lemmas \ref{hfLem} and \ref{lfLem}, we have a sufficient condition for exponential instability
of a periodic traveling wave of the gKdV to long wavelength transverse perturbations in the gKP equation.
This is the content of our main theorem, which we now state.

\begin{thm}\label{orientation}
Let $u=u(\cdot;a,E,c)$ be a periodic traveling wave of the gKdV equation \eqref{gkdv}.  Then $u$ is
spectrally unstable to transverse perturbations in the gKP equation \eqref{gkp} if the
product $\sigma\cdot\{T,M\}_{a,E}$ is positive.
\end{thm}

\begin{proof}
Clearly the negativity of the orientation index
\[
\sgn\left(D(0,k,1)\right)\cdot\lim_{\mu\to+\infty}\sgn\left(D(\mu,k,1)\right)
\]
for $0<|k|\ll 1$ implies the desired instability.  Since $PT-M^2>0$ by Jensen's inequality,
the result follows by Lemmas \ref{hfLem} and \ref{lfLem}.
\end{proof}

\begin{rem}
Notice the instability detected in Theorem \ref{orientation} is that of spectral
transverse instabilities to perturbations which are co-periodic in the $x$-direction
with low-frequency oscillations in the transverse ($y$) direction.  In the solitary
wave case, a general criterion for spectral transverse instability was recently
provided by Rousset ant Tzvetkov \cite{RT2}.  While it seems plausible that
such a criterion may exist when the underlying wave is spatially periodic,
we must note that the detected instability would not be to \emph{low}-frequency
oscillations in the transverse direction.  Indeed, in \cite{RT2} the transverse
frequency must be large enough that the kernel of a particular linear operator
be simple.  Once such a frequency $k_0\neq 0$ is found, an implicit function
type argument is used to prove the existence of a spectral curve $\mu\to(k(\mu),\mu)$
in a neighborhood of $\mu=0$ with $(k(0),0)=(k_0,0)$, thus proving spectral instability.
In the present work however, we are proving the existence of a map
$k\to\mu(k)$ defined for $|k|\ll 1$ such that $\mu(k)$ bounded away from zero for all $k\neq 0$
and such that $\mu(k)$ is an eigenvalue of \eqref{spec} for the associated small transverse
frequency $k$.
\end{rem}

%
As previously noted Theorem \ref{orientation} is only of interest when the underlying periodic wave is
a stable solution of the corresponding one dimensional problem.  In \cite{BrJK}, the nonlinear (orbital)
stability of such solutions to periodic perturbations in the gKdV equation was studied and the
Jacobian $\{T,M\}_{a,E}$ was seen to play a significant role.  Therein, it was shown that
such waves can be nonlinearly stable to such perturbations regardless of the sign of $\{T,M\}_{a,E}$,
assuming certain conditions on the perturbation and other geometric quantities related to the underlying
wave and the conserved quantities of the PDE flow.  Theorem \ref{orientation} however demonstrates
the direct influence the sign of this Jacobian determinant has on the stability to perturbations
in higher dimensional models (see also the recent work of Johnson \cite{J2} concerning the transverse instability
of periodic gKdV waves in the generalized Zakharov-Kuznetsov equations where the same Jacobian was seen
to control the low-frequency behavior of the corresponding Evans function).
In particular, we immediately have the following interesting (and seemingly unexpected) corollary.

\begin{corr}
A periodic traveling wave of the gKdV for which $\{T,M\}_{a,E}\neq 0$ can never be
spectrally stable to transverse perturbations in the gKP equation for both signs of dispersion.
\end{corr}

\begin{rem}
It should be noted that, unlike the ODE case, there are no general theorems insuring that spectral
instability implies nonlinear instability.  However, in the recent work \cite{RT1} it was
shown in the solitary wave context that, indeed, spectral transverse instability of KdV waves
in the KP-I equation (as described in the introduction) can be converted to a nonlinear instability
result.  It seems plausible that the methods utilized could apply in our case in order to convert
Theorem \ref{orientation} into a nonlinear instability result (at least in the case of periodic
transverse perturbations).  This would be an interesting direction for future investigation.
\end{rem}

We now point out several corollaries of Theorem \ref{orientation}.  In the case
of the KdV equation \eqref{eqn:kdv}, it is known that
all periodic traveling wave solutions are both spectrally stable to localized perturbations (see \cite{BD})
and nonlinearly (orbitally) stable to co-periodic perturbations (see \cite{BrJK} or \cite{J1}).  In particular,
the transverse stability of such solutions in the KP equation \eqref{eqn:KP} is of interest in this case.  Moreover,
it was shown in \cite{BrJK} that the Jacobian $\{T,M\}_{a,E}$ can be expressed as
\[
\{T,M\}_{a,E}=\frac{-T^2~V'\left(\frac{M}{T}\right)}{12~{\rm disc}\left(E-V(\cdot;a,c)\right)},
\]
where ${\rm disc}(R(\cdot))$ represents the discriminant of the polynomial $R$.
Since $V'$ is clearly strictly convex in this case, it follows by Jensen's inequality that
\[
V'\left(\frac{M}{T}\right)<\frac{1}{T}\int_0^TV'(u(x))dx=0.
\]
For an alternate proof of this fact, see \cite{J2}.
Moreover, notice that for any $(a,E,c)\in\RM^3$ for which \eqref{travelode} admits
a periodic solution of the KdV the equation $E=V(u;a,c)$ has three
solutions in $u$ and hence the discriminant must be positive.  Therefore,
we have that $\{T,M\}_{a,E}>0$ for all periodic traveling wave solutions of the KdV equation.
This proves the following corollary of Theorem \ref{orientation}.

\begin{corr}\label{kdvcor}
All periodic traveling wave solutions of the KdV are unstable to long wavelength transverse
perturbations in the KP equation when $\sigma>0$.
\end{corr}

It is interesting to note that it is known that all solitary wave solutions of the KdV
are transversely unstable in the KP equation when $\sigma>0$.  Thus, Corollary \ref{kdvcor}
seems to be somewhat expected.  Moreover, it turns out that the Galilean invariance of the KdV
implies we can always choose $a=0$, and hence (up to translation) the periodic traveling
waves of the KdV form only a {\em two} parameter family of solutions.  This family can be expressed
explicitly in terms of the Jacobi elliptic function as
\begin{equation*}
u(x,t)=u_0+12k^2\kappa^2\cn^2\left(\kappa\left(x+\left(8k^2\kappa^2-4\kappa^2+u_0\right)t\right),k\right),
\end{equation*}
where $u_0$ is an arbitrary parameter (taking the role of $E$) and $k$ is the elliptic modulus.
We refer to such a solution as a cnoidal wave solution of the KdV.
As a result of Corollary \ref{kdvcor}, it follows that all cnoidal wave solutions of the KdV
are unstable to long wavelength transverse perturbations in the KP-I equation.

We now move on to consider a periodic traveling wave solutions of the focusing mKdV equation
\[
u_t=u_{xxx}+u^2u_x
\]
with positive wave speed $c>0$.  When $a=0$, the corresponding traveling wave ODE admits
two distinct classes of periodic solutions: when $E>0$ the wave can again be expressed in terms
of the Jacobi elliptic function $\cn$, while when $E<0$ there either
exists no solution (if $|E|$ is sufficiently large) or the solution can be expressed in terms
of the Jacobi elliptic function $\dn$ and hence represents a dnoidal wave.  In the recent
work of \cite{BrJK}, it was shown that both sets of solutions are nonlinearly (orbitally) stable
to co-periodic perturbations, although the cnoidal solutions of sufficiently long wavelength
were shown to be (spectrally) unstable to periodic perturbations of large period\footnote{However, cnoidal waves
of smaller period seem spectrally stable to perturbations of sufficiently large period.} (see also \cite{DK}).  Moreover, the sign of
the Jacobian $\{T,M\}_{a,E}$ was analyzed for both the cnoidal
and dnoidal wave solutions of the focusing mKdV\footnote{In fact, the situation was analyzed without
the restriction of $a=0$ in which case all periodic traveling wave solutions of the focusing
mKdV can not be expressed simply in terms of a Jacobi elliptic function.  However, we
only consider the case $a=0$ here for simplicity.} and was seen to be positive for all dnoidal
solutions and negative for all cnoidal solutions.  As a result, we have the following corollary of
Theorem \ref{orientation}.

\begin{corr}
All cnoidal wave solutions of the focusing mKdV equation are unstable to long wavelength transverse
perturbations in the focusing mKP equation with $\sigma<0$,
while the dnoidal wave solutions are unstable to such perturbations when $\sigma>0$.
\end{corr}

Continuing, one can use the general elliptic function calculations of \cite{BrJK} in the case of a power-law
nonlinearity $f(u)=u^{p+1}$, $p\in\mathbb{N}$, to determine
the sign of the Jacobian $\{T,M\}_{a,E}$ for any periodic traveling wave solution of \eqref{gkdv}
in terms of moments of the background solution $u$ with respect to the density
\[
\frac{1}{\sqrt{E-V(\cdot;a,c)}}.
\]
As such, Theorem \ref{orientation} can be utilized to provide a transverse instability result for
such power-law nonlinearities.  In other cases, it seems that one must (in general) resort
to numerical methods to approximate $\{T,M\}_{a,E}$.

\section{Conclusions \& Discussion}

In this paper, we analyzed the spectral instability of a periodic traveling wave solutions of the generalized Korteweg-
de Vries equation to long wavelength transverse perturbations in the
generalized Kadomtsev-Petviashvili equation.  In particular, we constructed a seemingly nonstandard
orientation index by comparing the low and high frequency behavior of the periodic Evans function when
the transverse wave number $k$ is non-zero.  We found that in the high frequency limit, the Evans
function $D(\mu,k,1)$ converged to zero as $\mu\to\pm\infty$, which is insufficient to conclude a instability theory.
However, after taking into account higher order effects, it was found that by the periodicity
of the underlying wave and the resulting cancelation in averaging procedures that the
Evans function for non-zero transverse wave numbers favors a particular \emph{sign} as $\mu\to\pm\infty$ which is determined
precisely by the dispersion parameter $\sigma$: such a phenomenon seems to be new
in the literature.  Thus, an instability index follows
by comparing the sign of $\sigma$ with the value $D(0,k,1)$ for $k\neq 0$.  Utilizing the
methods of \cite{BrJ} and \cite{J2} then, we were able to explicitly compute the leading order
variation of the function $D(0,k,1)$ in $k$ in terms of a Jacobian from the traveling wave parameters
to the period and mass of the background solution.  This Jacobian was shown to be (generically) non-zero in the physically important cases
of the KdV and mKdV equations, and their resulting signs were inferred from the recent work of \cite{BrJK} immediately
yielding instability results in these cases.

It is interesting to note that the Jacobian arising in the low frequency expansion of the periodic Evans function
has already been seen to hold vital information concerning the stability of the periodic traveling
wave solutions of the gKdV.  Indeed, in \cite{BrJK} and \cite{J1} this Jacobian arose naturally in the
nonlinear stability analysis of such solutions to periodic perturbations in the gKdV equation, while
in \cite{J2} it was seen again to control the low frequency behavior of the periodic Evans function
when considering the spectral instability of a periodic gKdV wave to long wavelength transverse
perturbations in the generalized Zakharov-Kuznetsov equation (which also arises in plasma physics).
Thus, it would be very interesting to better understand the physical meaning of the Jacobian $\{T,M\}_{a,E}$
as this may better illuminate the stability theories described above.

Another interesting direction would be to complement the transverse \emph{instability} analysis in this
paper with a corresponding \emph{stability} theory.  That is, to derive sufficient conditions to guarantee
a periodic traveling wave solution of the gKdV is transversely stable to perturbations in the
gKP equation.  While this certainly may be possible using the Evans function techniques of this
paper, much more delicate analysis is needed.  In the context of the KdV equation,
one may be able to use the integrable structure of the KP-II equation to prove transverse
spectral stability.  Such techniques are prevalent in the solitary wave theory and have recently been
employed in the periodic wave setting to prove spectral stability
to localized perturbations for several model equations (see \cite{BD}, \cite{BDN}, and \cite{NB}).
When this integrable structure does not exist, however,
it may be more natural to consider a variational characterization of the stability problem and extend the
methods of \cite{GSS}, \cite{J1}, and \cite{BrJK}.

We also note that, after the completion of this work, our attention was brought to the
recent work of Rousset and Tzvetkov \cite{RT3} in which the authors considered the
transverse spectral instability of KdV solitary waves to perturbations in the KP-I equation.
In contrast to the ODE techniques utilized in this paper, the authors
use variational techniques to present a rather elegant and simple approach
relying only on properties of the differential operators involved which are rather easy to check (due to the
self-adjointness of the operators involved).  We believe it would be interesting as a future
direction of study to see if these techniques could apply in the case where the underlying
wave is spatially periodic and to compare the results to those derived in this paper.
Such a comparison would hopefully help illuminate the mechanism behind the instability.
On the other hand, the ODE  
techniques used in this paper are quite robust allowing for straightforward
numerical implementation (plotting Evans curves and computing winding numbers)
and applying not only to equations with a Hamiltonian like structure but also to
more complicated situations arising in the context of systems of nonlinear conservation laws.
Moreover, the degeneracy in the high frequency analysis is of independent interest adding new techniques
to the tool box in the study of nonlinear dispersive waves utilizing asymptotic
tracking/reduction results familiar from shock wave analysis in the conservation law setting.
%

Finally, as pointed out in the text the high frequency analysis conducted in this paper translates
directly to the solitary wave setting, and hence opens the door to an analogous instability theory
using Evans function techniques.  In fact, the high frequency limit seems easier to discern in the
solitary wave case as one does not have to deal with averaging effects of the coefficient
functions.  While we have not yet conducted the relevant low frequency analysis,
this could provide valuable insights concerning the transverse instability of a solitary traveling
wave of the gKdV in the gKP equation.

\medbreak
{\bf Acknowledgement.}
Thanks to Jared Bronski, Bj\"{o}rn Sandstede, Todd Kapitula, and Bernard Deconinck for
many useful conversations regarding the early stages of this work.  Their comments and suggestions
were invaluable to this project.  We would also like to thank a referee for bringing our attention
to the work of Rousset and Tzvetkov and for several helpful comments regarding the structure
of the KP equations.




\appendix

\section{A Block-Triangular Tracking Lemma}\label{s:track}
Consider an approximately block-triangular system
\begin{equation}
W'= A^p(x)W:= \bp M_1 & N \\ \delta \Theta & M_2 \ep(x,p)W,
\label{blockdiag}
\end{equation}
where $\Theta$ is a uniformly bounded matrix, $\delta(x)$ scalar,
and $p$ a vector of parameters,
satisfying a pointwise spectral gap condition
\begin{equation}
\min \sigma(\Re M_1)- \max \sigma(\Re M_2)
\ge \eta(x)>0
\, \text{\rm for all } x.
\label{gap}
\end{equation}
(Here as usual $\Re N:= (1/2)(N+N^*)$ denotes the
``real'', or symmetric part of $N$.)
Then, we have the following block-triangular version of the
{\it tracking/reduction lemma} of \cite{MaZ3,PZ}.
For related results, see \cite{HLZ}.

\begin{lemma}\label{reduction}
Consider a system \eqref{blockdiag} under the gap assumption
\eqref{gap}, with $\Theta$ uniformly bounded and
$\eta\in L^1_{\rm loc}$.
If $\sup (\delta/\eta)(x)$ is sufficiently small,
then there exists a unique bounded linear
transformation
\be\label{form}
S=\bp I & 0\\
\Phi & I\ep,
\ee
possessing the same regularity with respect to $p$
as do coefficients $M_j$ and $N$,
such that the change of coordinates $W=SZ$
converts the approximately triangular system \eqref{blockdiag} to an exactly block
triangular system
\be\label{exactsys}
Z'= \tilde A^p(x)Z:= \bp \tilde M_1 & \tilde N \\ 0 & \tilde M_2 \ep(x,p)Z
\ee
where
\be\label{tildecoeffs}
\tilde M_1:= M_1+ \Phi N, \qquad
\tilde M_2:= M_2- \Phi N, \qquad
\tilde N:= N,
\ee
with
\be\nn
\sup|\Phi| \le C \sup(\delta/\eta)
\ee
and
\ba\label{ptwise}
|\Phi(x)|&\le
C\int_x^{+\infty} e^{\int_y^x \eta(z)dz} \delta(y)dy,
\qquad
|\Phi(x)| \le C
\int_{-\infty}^{x} e^{\int_y^x -\eta(z)dz} \delta(y)dy.
\ea
where the constant $C$ depends only on the size of $\Theta$.
\end{lemma}

\begin{proof}
By the change of coordinates $x\to \tilde x$, $\delta \to \tilde \delta:=
\delta/\eta$ with
$d\tilde x/dx=\eta(x)$,
we may reduce to the case $\eta\equiv {\rm constant}= 1$ treated in \cite{MaZ3}.
Dropping tildes, we find by direct computation that
$Z=S^{-1}W$ satisfies \eqref{exactsys} for $S$ of form \eqref{form}
if and only if
\eqref{tildecoeffs} and
\be\nn
\Phi'=
(M_2 \Phi - \Phi M_1) + Q(\Phi),
\ee
where $Q$ is the quadratic matrix polynomial
$
Q(\Phi):=
\delta \Theta -  \Phi N \Phi. 
$
Viewed as a vector equation, this has the form
\be\nn
 \Phi'= \cM \Phi +  Q(\Phi),
\ee
with linear operator
$\cM \Phi:= M_2 \Phi - \Phi M_1$.
Note that a basis of solutions of the decoupled equation
$ \Phi'= \cM \Phi$
may be obtained as the tensor product $\Phi=\phi \tilde \phi^*$
of bases of solutions of $\phi'=M_2 \phi$ and
$\tilde \phi'= -M_1^* \tilde \phi$, whence we obtain from
\eqref{gap}
\be
e^{\cM z}\le Ce^{-\eta z}, \quad \hbox{\rm for }\; z>0,
\label{expbd}
\ee
or uniform exponentially decay in the forward direction.

Thus, assuming only that $\Phi$ is bounded at $-\infty$, we obtain
by Duhamel's principle the integral fixed-point equation
\be
\Phi(x)= \CalT \Phi(x):=
\int_{-\infty}^x e^{\cM (x-y)} Q(\Phi)(y)
\,dy.
\label{inteqn}
\ee
Using \eqref{expbd}, we find that $\CalT$ is a contraction
of order $O(\delta/\eta)$ for $\Phi$ on a ball of radius of
the same order, hence \eqref{inteqn} determines
a unique solution for $\delta/\eta$ sufficiently small, which,
moreover, is order $\delta/\eta$ as claimed.
(Here, we are using the normalization $\eta=1$.)
Finally, substituting
$Q(\Phi)=O(\delta+|\Phi|^2)=O(\delta)$ in
\eqref{inteqn}, we obtain
$$
|\Phi(x)|\le
C\int_{-\infty}^x e^{\eta (x-y)} \delta(y)
\,dy
$$
in $\tilde x$ coordinates, or, in the original $x$-coordinates,
\eqref{ptwise}.
Regularity with respect to parameters is inherited as usual
through the fixed-point construction via the Implicit Function Theorem.
\end{proof}

\begin{rem}
Although we do not use it here, an important observation of the above
proof is that for $\eta$ constant and $\delta$ decaying at exponential rate
strictly slower that $e^{-\eta x}$ as $x\to +\infty$,
we find from \eqref{ptwise} that $\Phi(x)$ decays like $\delta/\eta$
as $x\to +\infty$,
while if $\delta(x)$ merely decays monotonically as $x\to -\infty$, we
find that $\Phi(x)$ decays like $(\delta/\eta)$ as $x\to -\infty$.
\end{rem}

\begin{rem}\label{L1}
Though we do not use it here,
an important observation of \cite{MaZ3,PZ} is that
hypothesis \eqref{gap} of Lemma \ref{reduction} may be weakened to
\begin{equation}\nn
\min \sigma(\Re M_1^\varepsilon)- \max \sigma(\Re M_2^\varepsilon)
\ge \eta(x) +\alpha(x,p)>0
\end{equation}
with no change in the conclusions,
for any $\alpha$ satisfying a uniform $L^1$ bound
$|\alpha(\cdot,p)|_{L^1}\le C_1$.
(Substitute
$e^{\cM x}\le Ce^{C_1}e^{-\eta z}$
for \eqref{expbd}, with no other change in the proof.)
This allow us to neglect commutator terms in some of the
more delicate applications of tracking: for example, the
high-frequency analysis of \cite{MaZ3}.
\end{rem}

\begin{rem}\label{periodic}
In the special case that $A^p$ is $T$-periodic in $x$, we obtain by uniqueness
that $\Phi$ is $T$-periodic in $x$ as well.
\end{rem}

\section{Variation of Parameters Calculation}
In this appendix our goal is to justify equation \eqref{eqn:vinv}, which was seen to be a crucial
step in the low-frequency analysis of Section 3.  For brevity, however, we will consider only the
most difficult case.  To begin, define the scalar valued function $A_0(x):= e_1^\dag \WM(x,0,0)^{-1}e_4$,
where $\WM(x,0,0)$ is defined in \eqref{eqn:solnmatrix} and $\dag$ represents the vector adjoint,
and note from \eqref{eqn:vinv} we wish to prove
$A_0=-\int_0^x\int_0^su_E(z)dzds$.  By definition, we see that
\[
A_0=\left(u_{axx}u_{Ex}-u_{ax}u_{Exx}\right)\phi+\left(u_a u_{Exx}-u_{axx}u_E\right)\phi_x
         +\left(u_{ax}u_E-u_a u_{Ex}\right)\phi_{xx}.
\]
Using \eqref{quad1}, we have that
\begin{align*}
u_a u_{Exx}-u_{axx}u-E &= -u_aV''(u)u_E+u_E\left(V''(u)u_a-1\right)\\
&=-u_E.
\end{align*}
Noting that $\partial_x\left(u_a u_{Ex}-u_{ax}u_E\right)=u_au_{Exx}-u_{axx}u_E$, it follows
that
\[
u_{ax}u_E-u_a u_{Ex}=\int_0^xu_E(s)ds.
\]
Moreover, using \eqref{quad1} along with the fact that
\[
u_x\left(u_{ax}u_E-u_au_{Ex}\right)=uu_E-u_a,
\]
it follows that
\begin{align*}
u_x\left(u_{ax}u_{Exx}-u_{axx}u_{Ex}\right)&=-u_{Ex}u_x-V''(u)\left(uu_E-u_a\right)\\
&=-u_{Ex}u_x-V''(u)u_x\left(u_{ax}u_E-u_au_{Ex}\right)\\
&=-u_{Ex}u_x-V''(u)u_x\int_0^xu_E(s)ds.
\end{align*}
Therefore, using the above equalities along with the definition of the function $\phi$, we have
\begin{align*}
A_0&=-x\int_0^xu_E(s)ds\\
&+\left(\left(u_{Ex}+V''(u)\int_0^xu_E(s)ds\right)u_x-u_Eu_{xx}+u_{xxx}\int_0^xu_E(s)ds\right)\int_0^xsu_E(s)ds\\
&+\left(-\left(u_{Ex}+V''(u)\int_0^xu_E(s)ds\right)u_E-u_Eu_{Ex}+u_{Exx}\int_0^xu_E(s)ds\right)\int_0^xsu_s(s)ds.
\end{align*}
Now, as above one can show
\begin{align*}
\left(u_{Ex}+V''(u)\int_0^xu_E(s)ds\right)u_x&-u_Eu_{xx}+u_{xxx}\int_0^xu_E(s)ds\\
&=u_{Ex}u_x-u_{E}u_{xx}+\left(V''(u)u_x+u_{xxx}\right)\int_0^xu_E(s)ds\\
&=1
\end{align*}
since $u_{xxx}=-V''(u)u_x$ by \eqref{gkdv}.  Similarly, it follows that
\[
-\left(u_{Ex}+V''(u)\int_0^xu_E(s)ds\right)u_E-u_Eu_{Ex}+u_{Exx}\int_0^xu_E(s)ds=0
\]
and hence
\begin{align*}
A_0(x)&=-x\int_0^xu_E(s)ds+\int_0^xsu_{E}(s)ds\\
&=\int_0^x(s-x)u_E(s)ds\\
&=-\int_0^x\int_0^su_E(z)dzds
\end{align*}
as claimed.  The rest of the derivation of equation \eqref{eqn:vinv} is handled similarly, although
the necessary calculations are considerably simpler.



\begin{thebibliography}{GMWZ2}

{\footnotesize

\bibitem[APS]{APS} J. C. Alexander, R. L. Pego, and R. L. Sachs, \emph{On the transverse instability of solitary waves in the Kadomtsev-Petviashvili equation},
Phys. Lett. A., 226 (1997).

\bibitem[Be1]{Be1} T.B. Benjamin, \emph{Instability of periodic wavetrains in nonlinear dispersive systems},  Proc. Roy. Soc. A , 299  (1967).

\bibitem[Be2]{Be2} T. B. Benjamin, \emph{The stability of solitary waves}, Proc. Roy. Soc. (London) Ser. A, 328 (1972). 

\bibitem[BF]{BF} T. B. Benjamin and J. E. Feir, \emph{The disintegration of wave trains on deep water. Part 1. Theory.},
J. Fluid Mech. 27 (1967).

\bibitem[BD]{BD} N. Bottman and B. Deconinck, \emph{KdV cnoidal waves are linearly stable}, preprint.

\bibitem[BDN]{BDN} N. Bottman, B. Deconinck, and Michael Nivala, \emph{Elliptic solutions of the defocusing NLS equation are stable},
preprint.

\bibitem[BrJ]{BrJ} J. C. Bronski and M. Johnson, \emph{The modulational instability for a generalized Korteweg-de Vries equation},
ARMA, DOI 10.1007/s00205-009-0270-5 (2009).


\bibitem[BrJK]{BrJK}  J. C. Bronski, M. Johnson, and T. Kapitula, \emph{An index theorm for the stability of periodic traveling
waves of KdV type}, preprint.

\bibitem[Bo]{Bo} J. L. Bona, \emph{On the stability theory of solitary waves}, Proc. Roy. Soc. (London) Ser. A, 344 (1975), no. 1638. 

\bibitem[DK]{DK} B. Deconinck and T. Kapitula, \emph{On the orbital (in)stability of spatially periodic stationary solutions
of generalized Korteweg-de Vries equations}, preprint.


\bibitem[G1]{G1} R. A. Gardner, \emph{On the structure of the spectra of periodic travelling waves}, J. Math. Pures Appl. (9), 72 
(1993), no. 5.

\bibitem[G2]{G2} R. A. Gardner, \emph{Spectral analysis of long wavelength periodic waves and applications}, J.
Reine Angew. Math. 491 (1997).

\bibitem[GSS]{GSS} M. Grillakis, J. Shatah, and W. Strauss, \emph{ Stability theory of solitary waves in the presence of symmetry. I,II},
             J. Funct. Anal.  (1987), no. 74.

\bibitem[H]{H} D. Henry. Geometric theory of semilinear parabolic equations, volume 840 of
Springer Lecture Notes in Math. Springer-Verlag, 1981.

\bibitem[HK]{HK} M. H\v{a}r\v{a}gu\c{s} and T. Kapitula, \emph{On the spectra of periodic waves for infinite-dimensional Hamiltonian systems},
Physica D., 237, no. 20 (2008).  


\bibitem[HLZ]{HLZ} J. Humpherys, O. Lafitte, and K. Zumbrun,
{\it Stability of viscous shock profiles in the high Mach number
limit,} to appear, CMP (2009).

\bibitem[J1]{J1} M. Johnson, \emph{Nonlinear stability of periodic traveling wave solutions of the generalized Korteweg-de Vries equation},
SIMA, 41 no. 5 (2009).

\bibitem[J2]{J2} M. Johnson, \emph{The Transverse Instability of Periodic Waves in Zakharov-Kuznetsov Type Equations},
SAM, DOI: 10.1111/j.1467-9590.2009.00473.x (2010).

\bibitem[KP]{KP} B. B. Kadomtsev and V. I. Petviashvili, \emph{On the stability of solitary waves in weakly dispersive media},
Sov. Phys. Dokl., 15 (1970).

\bibitem[KTN]{KTN} T. Kataoka, M. Tsutahara, and Y. Negoroo, \emph{Transverse instability of solitary waves in the generalized Kadomtsev-Petviaschvili equation},
Phys. Rev. Lett., 84 (2000).

\bibitem[L]{L} M. J. Lighthill, \emph{Contributions to the theory of waves in non-linear dispersive systems},
J. Inst. Math. Applic. 1 (3) (1965).


\bibitem[MaZ3]{MaZ3} C. Mascia and K. Zumbrun,
{\it Pointwise Green function bounds for shock profiles of
systems with real viscosity.}
Arch. Ration. Mech. Anal.  169  (2003),  no. 3, 177--263.

\bibitem[NB]{NB} M. Nivala and B. Deconinck, \emph{Periodic finite-genus solutions of the KdV equation are orbitally stable}, preprint.

\bibitem[NZ]{NZ}  T. Nguyen and K. Zumbrun, \emph{Long-time stability of large-amplitude noncharacteristic
boundary layers for hyperbolic parabolic systems}, to appear, J. Math. Pure Appl.; preprint (2008).


\bibitem[PJ]{PJ} H. R. Pakzad and K. Javidan, \emph{Solitons of the Kadomstev-Petviashvili (KP) and modified KP (mKP) equations
for dust acoustic solitary waves in dusty plasmas with variable dust charge and nonthermal ions}, J. Phys.: Conference
Series, 96 (2008).

\bibitem[PW]{PW} R. L. Pego and M. I. Weinstein, \emph{Eigenvalues and instabilities of solitary waves}, Philos. Trans. Roy. Soc. London Ser. A.,
340 (1992), no. 1656. 

\bibitem[PZ]{PZ} R. Plaza and K. Zumbrun, \emph{An Evans function approach to
spectral stability of small-amplitude shock profiles}, J. Disc. and
Cont. Dyn. Sys. 10 (2004).

\bibitem[RT1]{RT2} F. Rousset and N. Tzvetkov, \emph{A simple criterion of transverse linear instability for solitary waves}, preprint.

\bibitem[RT2]{RT3} F. Rousset and N. Tzvetkov, \emph{Transverse nonlinear instability for two-dimensional dispersive models},
Ann. Inst. H. Poincaré Anal. Non Linéaire 26 (2009).

\bibitem[RT3]{RT1} F. Rousset and N. Tzvetkov, \emph{Transverse nonlinear instability of solitary waves for some Hamiltonian PDE's},
J. Math. Pure et App., 90 (2008).

\bibitem[TRR]{TRR} S. K. Turitsyn, J. Jull Rasmussen, and M. A. Raadu, \emph{Stability of weak double layers},
Royal Institute of Technology, Stockholm, TRITA-EPP-91 -01 (1991).

\bibitem[W]{W} Whitham, G. B. \emph{Non-linear dispersion of water waves}, J. Fluid Mech., 27 (1967).

\bibitem[Za]{Za} V. Zakharov, \emph{Instability and nonlinear oscillations of solitons}, JEPT Lett., 22 (1975).

\bibitem[Z6]{Z6} K. Zumbrun,
{\it Stability of detonations in the ZND limit}, preprint (2009).









}

\end{thebibliography}
\end{document}